\documentclass[11pt]{article}
\usepackage{amsfonts}
\usepackage{amsmath}
\usepackage{amssymb}
\usepackage[T1]{fontenc}
\usepackage{latexsym,amssymb,amsmath,amsfonts,amsthm}
\newcommand{\R}{{\mat R}}
\newcommand{\Z}{{\mat Z}}

\newcommand{\C}{{\mat C}}
\newcommand{\ds}{\displaystyle}
\newcommand{\no}{\nonumber}
\newcommand{\be}{\begin{eqnarray}}
\newcommand{\ben}{\begin{eqnarray*}}
\newcommand{\en}{\end{eqnarray}}
\newcommand{\enn}{\end{eqnarray*}}
\newcommand{\ba}{\backslash}
\newcommand{\pa}{\partial}

\newcommand{\ov}{\overline}
\newcommand{\curl}{{\rm curl\,}}
\newcommand{\Grad}{{\rm Grad\,}}
\newcommand{\dive}{{\rm div\,}}
\newcommand{\I}{{\rm Im\,}}
\newcommand{\Rt}{{\rm Re\,}}
\newcommand{\g}{\gamma}
\newcommand{\G}{\Gamma}
\newcommand{\eps}{\epsilon}
\newcommand{\vep}{\varepsilon}
\newcommand{\Om}{\Omega}
\newcommand{\om}{\omega}

\newcommand{\al}{\alpha}
\newcommand{\bt}{\beta}
\newcommand{\wi}{\widehat}
\newcommand{\ti}{\times}
\newcommand{\wid}{\widetilde}
\newcommand{\de}{\delta}

\newcommand{\na}{\nabla}
\newcommand{\mat}{\mathbb}
\newcommand{\se}{\setminus}
\newcommand{\la}{\lambda}

\newtheorem{theorem}{Theorem}[section]
\newtheorem{lemma}[theorem]{Lemma}

\newtheorem{remark}[theorem]{Remark}

\begin{document}
\renewcommand{\theequation}{\arabic{section}.\arabic{equation}}
\begin{titlepage}
\title{\bf Inverse electromagnetic scattering problems by a doubly periodic structure}
\author{Jiaqing Yang and Bo Zhang\\
LSEC and Institute of Applied Mathematics\\
Academy of Mathematics and Systems Science\\
Chinese Academy of Sciences\\
Beijing 100190, China\\
({\sf jiaqingyang@amss.ac.cn} (JY),\
{\sf b.zhang@amt.ac.cn} (BZ))}
\date{}
\end{titlepage}
\maketitle

\begin{abstract}
Consider the problem of scattering of electromagnetic waves by a doubly periodic structure.
The medium above the structure is assumed to be inhomogeneous characterized
completely by an index of refraction. Below the structure is a perfect conductor
or an imperfect conductor partially coated with a dielectric.
Having established the well-posedness of the direct
problem by the variational approach, we prove the uniqueness of the inverse problem, that is,
the unique determination of the doubly periodic grating with its physical property
and the index of refraction from a knowledge of the scattered near field by a countably infinite
number of incident quasi-periodic electromagnetic waves. A key ingredient in our proofs is
a novel mixed reciprocity relation derived in this paper.

\vspace{.2in}
\end{abstract}

\section{Introduction}
\setcounter{equation}{0}

Scattering theory in periodic structures has many applications in micro-optics,
radar imaging and non-destructive testing. We refer to \cite{PR} for historical remarks and
details of these applications.
This paper is concerned with direct and inverse problems of electromagnetic scattering
by a doubly periodic structure. The medium above the structure is assumed to be inhomogeneous.
Below the structure is a perfect conductor which may be partially coated with a dielectric.

Let the doubly periodic structure be described by the doubly periodic surface
\ben
\G_1:=\{x\in\R^3\,|\,x_3=f(x_1,x_2)\},
\enn
where $f\in C^2(\R^2)$ is a $2\pi$-periodic function of $x_1$ and $x_2$:
\ben
f(x_{1}+2n_{1}\pi,x_2+2n_2\pi)=f(x_1,x_2)\quad\forall n=(n_1,n_2)\in\Z^2.
\enn
Assume that the medium above the structure $\G_1$ is filled with an inhomogeneous, isotropic,
conducting or dielectric medium of electric permittivity $\epsilon>0$, magnetic permeability $\mu>0$
and electric conductivity $\sigma\geq 0$.
Suppose the medium is non-magnetic, that is, the magnetic permeability $\mu$ is a fixed constant in
the region above $\G_1$ and the field is source free. Then the electromagnetic wave propagation
is governed by the time-harmonic Maxwell equations (with the time variation of the
form $e^{-i\om t},$ $\om>0$)
\ben\label{maxwell}
 \curl E-i\omega\mu H=0,\qquad\;\;\curl H+i\omega(\epsilon+i\sigma/\omega)E=0,
\enn
where $E$ and $H$ are the electric and magnetic fields, respectively. Suppose
the inhomogeneous medium is $2\pi$-periodic with respect to the $x_1$ and $x_2$ directions,
that is, for all $n=(n_1,n_2)\in\Z^2$,
\ben
\epsilon(x_1+2\pi n_1,x_2+2\pi n_2,x_3)=\epsilon(x_1,x_2,x_3),\qquad
\sigma(x_1+2\pi n_1,x_2+2\pi n_2,x_3)=\sigma(x_1,x_2,x_3).
\enn
Suppose above the structure $\G_1$ is another doubly periodic surface defined by
\ben
\G_0:=\{x\in\R^3\,|\,x_3=g(x_1,x_2)\},
\enn
where $g\in C^2(\R^2)$ is a $2\pi$-periodic function of $x_1$ and $x_2$:
\ben
g(x_{1}+2n_{1}\pi,x_{2}+2n_{2}\pi)=g(x_{1},x_{2})\quad\forall n=(n_{1},n_{2})\in\Z^2,
\enn
which separates the region above $\G_1$ into two parts:
\ben
\Om_0&:=&\{x\in\R^3\,|\,x_3>g(x_1,x_2)\},\\
\Om_1&:=&\{x\in\R^3\,|\,f(x_1,x_2)<x_3<g(x_1,x_2)\}.
\enn
Assume further that $\epsilon(x)=\epsilon_0,$ $\sigma=0$ for $x\in\Om_0$ (which
means that the medium above the layer is lossless) and that the doubly periodic surface
$\Gamma_1$ is a perfectly conductor coated partially with a dielectric.

Consider the scattering of the electromagnetic plane wave
\ben
E^i(x)=pe^{ik_{0}x\cdot d},\quad H^i(x)=re^{ik_{0}x\cdot d}
\enn
incident on the doubly periodic structure $\G_0$ from the top region $\Om_0$,
where $k_0=\sqrt{\epsilon_0\mu}\omega$ is the wave number,
$d=(\al_1,\al_2,-\bt)=(\cos\theta_1\cos\theta_2,\cos\theta_1\sin\theta_2,-\sin\theta_1)$
is the incident wave vector whose direction is specified by $\theta_1$ and $\theta_2$ with
$0<\theta_1\le\pi,\,0<\theta_2\leq 2\pi$ and the vectors $p$ and $r$ are polarization directions
satisfying that $p=\sqrt{{\mu}/{\vep_0}}(r\times d)$ and $r\bot d$.
The problem of scattering of time-harmonic electromagnetic waves in this model leads to the
following problem (the magnetic field $H$ is eliminated):
\be\label{1.1a}
\curl\curl E-k_0^2E=0 &&\mbox{in}\;\;\Om_0,\\ \label{1.1b}
\curl\curl E-k_0^2q(x)E=0 &&\mbox{in}\;\;\Om_1,\\ \label{1.1c}
\nu\times E|_{+}=\nu\times E|_{-},\quad\nu\times\curl E|_{+}
             =\la_0\nu\times\curl E|_{-} &&\mbox{on}\;\;\G_0,\\ \label{1.1d}
\nu\times E=0 &&\mbox{on}\;\;\G_{1,D},\\ \label{1.1e}
\nu\times\curl E-i\rho(\nu\times E)\times\nu=0 &&\mbox{on}\;\;\G_{1,I}
\en
where $q(x)=(\epsilon(x)+i{\sigma(x)}/{\omega})/\epsilon_0$ is the refractive index,
$\nu$ is the unit normal at the boundary, $E=E^i+E^s$ is the total field in
$\Om_0$ with $E^s$ being the scattered electric field,
$\G_1=\ov{\G}_{1,D}\cup\ov{\G}_{1,I}$,
$\la_0$ and $\rho$ are two positive constants.

We require the scattered field $E$ to be {\em $\al$-quasi-periodic} with respect to
$x_1$ and $x_2$ in the sense that $E(x_1,x_2,x_3)e^{-i\al\cdot x}$ is $2\pi$ periodic
with respect to $x_1$ and $x_2$, where $\al=(\al_1,\al_2,0)\in\R^3$.
It is further required that the scattered field $E$ satisfies the following
Rayleigh expansion radiation condition as $x_3\to+\infty$:
\be\label{RE}
E^s(x)=\sum_{n\in\Z^2}E_ne^{i(\al_n\cdot x+\beta_nx_3)},\quad x_3>g_+:=\max_{x_1,x_2}g(x_1,x_2),
\en
where $\al_n=(\al_1+n_1,\al_2+n_2,0)\in\R^3$, $E_n=(E_n^{(1)},E_n^{(2)},E_n^{(3)})\in\C^3$
are the Rayleigh coefficients and
\ben
\beta_n=\begin{cases}
        (k_0^2-|\al_n|^2)^{1/2} &\mbox{if}\;|\al_n|^2\le k_0^2,\\
        i(|\al_n|^2-k_0^2)^{1/2} &\mbox{if}\;|\al_n|^2>k_0^2
        \end{cases}
\enn
with $i^2=-1$. From the fact that $\dive E^s(x)=0$ it is clear that
\ben
   \al_n\cdot E_n+ \beta_nE_n^{(3)}=0.
\enn
Throughout this paper we assume that $\beta_n\neq 0$ for all $n\in\Z^2$.

The direct problem is to compute the scattered field $E^s$ in $\Om_0$ and $E$ in $\Om_1$
given the incident wave $E^i$, the diffraction grating profiles $\G_0$ and
$\G_1$ with the corresponding boundary conditions and the refractive index $q(x)$.
Our inverse problem is to determine the grating profile $\G_1$ together with the impedance
coefficient $\rho$ in the case when the interface grating profile $\G_0$ is known
and the refractive index $q$ in the case when the grating surfaces $\G_0$
and $\G_1$ are known and flat, utilizing the knowledge of the incident wave $E^i$
and the total tangential electric field $\nu\times E$ on a plane $\G_h=\{x\in\R^3\,|\,x_3=h\}$
above the inhomogeneous layer.

Problems of scattering of electromagnetic waves by a doubly periodic structure have been
studied by many authors using both integral and variational methods.
The reader is referred to, e.g. \cite{AT,AB,B,BLW,D,DF,NS,Schmidt} for results on existence,
uniqueness, and numerical approximations of solutions to the direct problems.
Compared with the direct problem, not much attention has been paid to
inverse problems from doubly periodic structures although they are not only mathematically
interesting but have many important applications.
For the case when $\G_{1,I}=\emptyset$ and the medium above the periodic structure
$\G_1=\G_{1,D}$ is homogeneous,
the inverse scattering problem has been considered in \cite{AH,BZ,BZZ10}.
If the medium is lossy above the perfectly reflecting periodic structure,
Ammari \cite{AH} proved a global uniqueness result for the inverse problem
with one incident plane wave. If the medium is lossless above the perfectly
reflecting periodic structure, a local uniqueness result was obtained in \cite{BZ}
for the inverse problem with one incident plane wave by establishing a lower
bound of the first eigenvalue of the $\curl\curl$ operator with the
boundary condition (\ref{1.1d}) in a bounded, smooth convex domain in $\R^3.$
The stability of the inverse problem was also studied in \cite{BZ}.
Recently in \cite{BZZ10}, for the class of perfectly reflecting doubly
periodic polyhedral structures global uniqueness results have been
established in \cite{BZZ10} for the inverse problem in the case of lossless
medium above the structure, using only a minimal number (though unknown) of incident plane waves .
Further, for a general Lipschitz, bi-periodic, partly coated structure $\G_1$
a global uniqueness result was proved in \cite{HQZ} for the inverse problem
in the case of a lossless, homogeneous medium above the structure, using infinitely
many incident dipole sources.

On the other hand, for the case when $\G_{1,I}=\emptyset$ (i.e. $\G_1=\G_{1,D}$),
$\lambda_0=1$ and the grating surfaces $\G_0$ and $\G_1$ are known and flat,
a global uniqueness result was obtained in \cite{HYZ} for reconstructing
the refractive index $q$, using all electric dipole incident waves
(see \cite{KA} for the corresponding result in the 2D case).

In this paper, we prove global uniqueness results for the inverse problem of
recovering a general smooth bi-periodic profile with a mixed boundary condition
and a known bi-periodic interface from a knowledge of near field measurements
above the known interface with a countably infinite number of quasi-periodic incident waves
$E^i(x;m)=(1/k^2_0)\curl\curl[p\exp({i\al_m\cdot x-i\beta_{m}x_3})],$ $m=(m_1,m_2)\in\Z^2.$
Further, we also establish a global uniqueness result for the inverse problem of
determining the refractive index $q$ which depends on only one direction ($x_1$ or $x_2$)
for the case when $\G_{1,I}=\emptyset$ (i.e. $\G_1=\G_{1,D}$)
and the grating surfaces $\G_0$ and $\G_1$ are known and flat,
using a countably infinite number of quasi-periodic incident waves $E^i(x;m).$
This is an improvement to the result of \cite{HYZ}.
A key ingredient in our proofs is a novel mixed reciprocity relation derived
in this paper for bi-periodic structures.

The rest of the paper is organized as follows. In Section \ref{sec2}, we introduce
some suitable quasi-periodic function spaces and the Dirichlet-to-Neumann map
on an artificial boundary above the structure. The problem (\ref{1.1a})-(\ref{RE}) is
then reduced to a boundary value problem in a truncated domain.
In Section \ref{sec3}, we establish the well-posedness of the
scattering problem (\ref{1.1a})-(\ref{RE}), employing a variational approach.
Section \ref{sec4} is devoted to the inverse problems.
In Subsection \ref{sec41} novel mixed reciprocity relations are established for
doubly periodic structures, which play a key role in the proofs of
the uniqueness results for our inverse problems.
Subsection \ref{sec42} is devoted to the unique determination of the
doubly periodic grating profile $\G_1$ with its physical property,
where it is assumed that the interface $\G_0$ is known and the refractive index $q$
is a known constant.
Subsection \ref{sec43} is concerned with the unique reconstruction of the refractive index $q$,
where we only consider the case when the shape of the two grating profiles is known and flat,
which improves the result in \cite{HYZ}.

\section{Quasi-periodic function spaces}\label{sec2}
\setcounter{equation}{0}

In this section we introduce some function spaces needed in the study of our problems.
Due to the periodicity of the problem, the original problem can be reduced to
a problem in a single periodic cell of the grating profiles. To this end and
for the subsequent analysis, we use $\G_j,\Om_j$ ($j=0,1$), $\G_{1,D},\G_{1,I}$
and $\G_h$ for $h\in\R$ again to denote the single periodic part (i.e. in the
range $0<x_1,x_2<2\pi$) of the corresponding notations defined in the last section.
We also need the notation
\ben
\Om_h=\{x\in\R^3\,|\,0<x_1,x_2<2\pi,\;f(x_1,x_2)<x_3<h\}
\enn
for $h>\max\{f(x')\,|\,x'\in\R^2\}.$

We now introduce some vector quasi-periodic Sobolev spaces. Let
\ben
H(\curl,\Om_h)=\{E(x)=\sum_{n\in\Z^2}E_n(x_3)\exp({i\al_n\cdot x})\,|\,
                E\in(L^2(\Om_h))^3,\,\curl E\in(L^2(\Om_h))^3\}
\enn
with the norm
\ben
||E||_{H(\curl,\Om_h)}^2=||E||_{L^2(\Om_h)}^2+||\curl E||_{L^2(\Om_h)}^2
\enn
Note that the $\al$-quasi-periodic space $H(\curl,\Om_b)$ is a subset of the classical vector
space $\mathbb{H}(\curl,\Om_b)$ defined by
\ben
\mathbb{H}(\curl,\Om_b)=\{E\in(L^2(\Om_b))^3\,|\,\curl E\in(L^2(\Om_b))^3\}
\enn
with the norm $||E||^2_{\mathbb{H}(\curl,\Om_b)}=||E||^2_{L^2(\Om_b)}+||\curl E||^2_{L^2(\Om_b)}.$
Further, it was shown in \cite{BLW} that $H(\curl,\Om_b)$ can be characterized as
\ben
H(\curl,\Om_b)&=&\{E\in\mathbb{H}(\curl,\Om_b)\,|\,e^{2\pi i\al_1}E(0,x_2,x_3)
                    \times e_1=E(2\pi,x_2,x_3)\times e_1,\\
              &&\qquad e^{2\pi i\al_2}E(x_1,0,x_3)\times e_2=E(x_1,2\pi,x_3)\times e_2\},
\enn
where $e_1=(1,0,0)$ and $e_2=(0,1,0).$

To deal with the mixed boundary conditions (\ref{1.1d}) and (\ref{1.1e}), we introduce the
subspace of $H(\curl,\Om_{h}):$
\ben
X:=\{E\in H(\curl,\Om_h)\,|\,\nu\times E|_{\G_{1,D}}=0,\;
     \nu\times E|_{\G_{1,I}}\in L_t^2(\G_{1,I})\}
\enn
with the norm
\ben
||E||_X^2=||E||_{H(\curl,\Om_h)}^2+||\nu\times E||_{L_t^2(\G_{1,I})}^2
\enn
where $L_t^2(\G_{1,I})=\{E\in(L^2(\G_{1,I}))^3\,|\,\nu\cdot E=0\;\mbox{on}\;\G_{1,I}\}$.

For $x^\prime=(x_1,x_2,h)\in\G_h$, $s\in\R$ define
\ben
H_t^s(\G_h)&=&\{E(x^\prime)=\sum_{n\in\Z^2}E_n\exp(i\al_n\cdot x^\prime)\,|\,
              E_n\in\C^3,\,e_3\cdot E=0,\\
           &&\qquad \|E\|^2_{H^s(\G_h)}=\sum_{n\in\Z^2}(1+|\al_n|^2)^s|E_n|^2<+\infty\}\\
H_t^s(\dive,\G_h)&=&\{E(x^\prime)=\sum_{n\in\Z^2}E_n\exp(i\al_n\cdot x^\prime)\,|\,
                     E_n\in\C^3,\,e_3\cdot E=0,\\
                 &&\quad ||E||^2_{H^s(\dive,\G_h)}=\sum_{n\in\Z^2}(1+|\al_n|^2)^s
                        (|E_n|^2+|E_n\cdot\al_n|^2)<+\infty\}\\
H_t^s(\curl,\G_h)&=&\{E(x^\prime)=\sum_{n\in\Z^2}E_n\exp(i\al_n\cdot x^\prime)\,|\,
                     E_n\in\C^3,\,e_3\cdot E=0,\\
                 &&\quad ||E||^2_{H^s(\curl,\G_h)}=\sum_{n\in\Z^2}(1+|\al_n|^2)^s
                     (|E_n|^2+|E_n\times\al_n|^2)<+\infty\}
\enn
and write $L_t^2(\G_h)=H_t^0(\G_h).$
We have the duality result:
\ben
(H_t^s(\dive,\G_h))'=H_t^{-s-1}(\curl,\G_h).
\enn
Recalling the trace theorem on $\mathbb{H}(\curl,\Om_h)$, we have
\ben
H_t^{-1/2}(\dive,\G_h)=\{e_3\times E|_{\G_h}\,|\,E\in H(\curl,\Om_h)\}
\enn
and that the trace mapping from $H(\curl,\Om_h)$ to $H_t^{-1/2}(\dive,\G_h)$ is continuous and
surjective (see \cite{BCS} and the references there).
We also need the trace space $Y(\G_0)$ and its duality space $Y(\G_0)'$:
\ben
Y(\G_0)&=&\{f\in H_t^{-1/2}(\G_0)\,|\,\nabla_{\G_0}\cdot f\in H^{-1/2}(\G_0)\}\\
Y(\G_{0})'&=&\{f\in H_t^{-1/2}(\G_0)\,|\,\nabla_{\G_0}\times f\in H^{-1/2}(\G_0)\},
\enn
where $\nabla_{\G_0}$ denotes the surface gradient on $\G_0.$
Note that the trace space $Y(\G_0)$ can also be defined as follows (see \cite{Chen}
and \cite[p. 58-59]{M}):
\ben
Y(\G_0)&=&\{f\in(H^{-1/2}(\G_0))^3\;|\;\mbox{there exists}\,E\in H(\curl,\Om_h\ba\ov{\Om}_1)\\
       &&\;\mbox{with}\,\nu\times E=f\;\mbox{on}\;\G_0,\;\nu\times E=0\;\mbox{on}\;\G_h\}
\enn
for $h>g_+,$ with norm
\ben
\|f\|_{Y(\G_0)}=\inf\{\|E\|_{H(\curl,\Om_h\ba\ov{\Om}_1)}\;|\;E\in H(\curl,\Om_h\ba\ov{\Om}_1),
\;\nu\times E\big|_{\G_0}=f,\;\nu\times E\big|_{\G_h}=0\}.
\enn

We assume throughout this paper that $q$ satisfies the following conditions:
\begin{description}
\item ({\bf A1}) $q\in C^1(\ov{\Om}_1)$ and $q(x)=1$ for $x\in\Om_0$;
\item ({\bf A2}) $\I[q(x)]\geq0$ for all $x\in\ov{\Om}_1$ if $\G_{1,I}\neq\emptyset$ and
                 $\I[q(x_0)]>0$ for some $x_0\in\ov{\Om}_1$ if $\G_{1,I}=\emptyset$;
\item ({\bf A3}) $\Rt[q(x)]\ge\g$ for all $x\in\ov{\Om}_1$ for some positive constant $\g$.
\end{description}

\section{The direct scattering problem}\label{sec3}
\setcounter{equation}{0}

In this section we will establish the solvability of the scattering problem
(\ref{1.1a})-(\ref{RE}), employing the variational method.
To this end, we propose a variational formulation of the scattering problem in
a truncated domain by introducing a transparent boundary condition on $\G_h$
for $h>g_+$.

Let $x^\prime=(x_1,x_2,h)\in\G_h$ for $h>g_+$.
For $\wi{E}\in H_t^{-1/2}(\dive,\G_h)$ with
$$
\wi{E}(x^\prime)=\sum_{n\in\Z^2}\wi{E}_n\exp(i\al_n\cdot x^\prime)
$$
define the Dirichlt-to-Neumann map
$\mathcal{R}:\,H_t^{-1/2}(\dive,\G_h)\rightarrow H_t^{-1/2}(\curl,\G_h)$ by
\ben\label{DtoN}
(\mathcal{R}\wid{E})(x^\prime)=(e_3\times\curl E)\times e_3\quad{\rm on}\;\;\G_h,
\enn
where $E$ satisfying the Rayleigh expansion condition (\ref{RE}) is the unique
quasi-periodic solution to the problem
\ben
\curl\curl E-k^2E=0\quad{\rm for}\;x_3>h,\qquad
\nu\times E=\wid{E}(x^\prime)\quad{\rm on}\;\G_h.
\enn
The map $\mathcal{R}$ is well-defined and can be used to replace the radiation condition
(\ref{RE}) on $\G_h$. The scattering problem (\ref{1.1a})-(\ref{RE}) can then be
transformed into the following boundary value problem in the truncated domain $\Om_h$:
\be\label{s3.1}
\curl\curl E-k_0^2E=0 &&\mbox{in}\;\;\Om_0,\\ \label{s3.2}
\curl\curl E-k_0^2qE=0 &&\mbox{in}\;\;\Om_1,\\ \label{s3.3}
\nu\times E|_{+}-\nu\times E|_{-}=f_1,\;\;
  \nu\times\curl E|_{+}-\la_0\nu\times\curl E|_{-}=f_2 &&\mbox{on}\;\;\G_0,\\ \label{s3.4}
\nu\times E=f_3\quad\text{on}\quad\G_{1,D},\quad
  \nu\times\curl E-i\rho(\nu\times E)\times\nu=f_4 &&\mbox{on}\;\;\G_{1,I},\\ \label{s3.5}
(\curl E)_T-\mathcal{R}(e_3\times E)=0 &&\mbox{on}\;\;\G_h,
\en
where $f_1=-\nu\times E^i|_{\G_0}\in Y(\G_0),\;f_2=-\nu\times\curl E^i|_{\G_0}\in Y(\G_0)'$,
$f_3=f_4=0$ and, for any vector $F$, $(F)_T=(\nu\times F)\times\nu$ denotes its
tangential component on a surface.

\begin{remark}\label{re1} {\rm
In the case when $k_0^2q(x)\equiv k_1^2$ is a constant and the incident field
is given by the electric dipole source $E^i(x)=\mat{\wi{G}}_1(x,y_0)r$ for
$y_0\in\Om_1$ and $r\in\R^3$ (e.g. in the problem (\ref{l41.a})-(\ref{l41.g})
of Lemma \ref{le41}), we have $f_1=\nu\times E^i|_{\G_0}\in Y(\G_0),$
$f_2=\la_0\nu\times\curl E^i|_{\G_0}\in Y(\G_0)'$,
$f_3=-\nu\times E^i|_{\G_{1,D}}\in Y(\G_{1,D}),$
$f_4=-\nu\times\curl E^i|_{\G_{1,I}}+i\rho(\nu\times E^i)\times\nu|_{\G_{1,I}}\in L_t^2(\G_{1,I})$
in the problem (\ref{s3.1})-(\ref{s3.5}),
where $Y(\G_{1,D})$ is defined in the same way as $Y(\G_0)$ with $\G_0$ replaced by $\G_{1,D}$
(see \cite{HQZ}).
}
\end{remark}

Define
\ben
Y&:=&\{E\in H(\curl,\Om_1)\cap H(\curl,\Om_h\ba\ov{\Om}_1)\,|\,\nu\times E|_{\G_{1,D}}=f_3,\\
&&\qquad\nu\times E|_{\G_{1,I}}\in L_t^2(\G_{1,I}),\;
        \nu\times E|_{+}-\nu\times E|_{-}=f_1\;\mbox{on}\;\G_0\}
\enn
with the norm
\ben
\|E\|^2_Y=\|E\|^2_{H(\curl,\Om_1)}+\|E\|^2_{H(\curl,\Om_h\ba\ov{\Om}_1)}
          +||\nu\times E||_{L_t^2(\G_{1,I})}^2.
\enn
Then the variational formulation for the problem (\ref{s3.1})-(\ref{s3.5}) is given as follows:
find $E\in Y$ such that
\be\label{s3.6}
A(E,F)=B(F)\qquad\forall\;F\in X,
\en
where
\ben
A(E,F)&:=&\la_0\int_{\Om_1}(\curl E\cdot\curl\ov{F}-k^2_0qE\cdot\ov{F})dx\\
      &&+\int_{\Om_h\ba\ov{\Om}_1}(\curl E\cdot\curl\ov{F}-k^2_0E\cdot\ov{F})dx\\
      &&-i\la_0\rho\int_{\G_{1,I}}E_T\cdot\ov{F}_Tds
        -\int_{\G_h}\mathcal{R}(\nu\times E)\cdot(\nu\times\ov{F})ds,\\
B(F)&:=&\int_{\G_0}f_2\cdot\ov{F}_Tds+\la_0\int_{\G_{1,I}}f_4\cdot\ov{F}_Tds.
\enn
If $f_1\in Y(\G_0)$, then there exits $\wid{E}_0\in H(\curl,\Om_h\ba\ov{\Om}_1)$ such that
$\nu\times\wid{E}_0|_{\G_0}=f_1$, $\nu\times\wid{E}_0|_{\G_h}=0$.
Similarly, for $f_3\in Y(\G_{1,D})$ there exists a function $E_1\in H(\curl,\Om_h)$
such that $\nu\times E_1|_{\G_{1,D}}=f_3,$ $\nu\times E_1|_{\G_h}=0$ and
$\nu\times E_1|_{\G_{1,I}}\in L^2_t(\G_{1,I})$. Let $\wid{E}=E-E_0-E_1,$
where $E_0$ is a function in $\Om_h$ satisfying that $E_0|_{\Om_h\ba\ov{\Om}_1}=\wid{E}_0$
and $E_0|_{\Om_1}=0.$ Then $\wid{E}\in X$ and the variational problem (\ref{s3.6})
is equivalent to the problem: find $\wid{E}\in X$ such that
\be\label{s3.7}
A(\wid{E},F)=\wid{B}(F)\qquad\forall\;F\in X,
\en
where $\wid{B}(F)=B(F)-A(E_0,F)-A(E_1,F).$

\begin{theorem}\label{thm3.1}
Assume that the conditions (\textbf{A1})-(\textbf{A3}) are satisfied.
Then the problem $(\ref{s3.1})-(\ref{s3.5})$ has a unique solution $E\in Y$ for any
$f_1\in Y(\G_0),\,f_2\in Y(\G_0)'$, $f_3\in Y(\G_{1,D})$ and $f_4\in L^2_t(\G_{1,I})$.
Furthermore, we have
\ben
||E||_Y\le C(||f_1||_{Y(\G_0)}+||f_2||_{Y(\G_0)'}
+||f_3||_{Y(\G_{1,D})}+||f_4||_{L^2_t(\G_{1,I})}),
\enn
where $C$ is a positive constant depending only on $\Om_h$.
\end{theorem}

\begin{proof}
It is enough to prove that the problem (\ref{s3.7}) has a unique solution $\wid{E}\in X$
with the required estimate.

We first prove the uniqueness of solutions. To this end, let $f_j=0,$ $j=1,2,3,4$
and let $F=\wid{E}$ in (\ref{s3.7}). Then $A(\wid{E},\wid{E})=0$, that is,
\ben
&&\la_0\int_{\Om_1}(|\curl\wid{E}|^2-k^2_0q|\wid{E}|^2)dx
  +\int_{\Om_h\ba\ov{\Om}_1}(|\curl\wid{E}|^2-k^2_0|\wid{E}|^2)dx\\
&&\qquad-i\la_0\rho\int_{\G_{1,I}}|\wid{E}_T|^{2}ds
  -\int_{\G_h}\mathcal{R}(\nu\times\wid{E})\cdot(\nu\times\ov{\wid{E}})ds=0.
\enn
Taking the imaginary part of the above equation and noting that
the imaginary part of the last integral in the above equation
is non-negative (see \cite[Equation (16)]{HQZ}), we deduce that
\be\label{s3.8}
k^2_0\int_{\Om_1}\I(q)|\wid{E}|^2dx+\rho\int_{\G_{1,I}}|\wid{E}_T|^2ds\le 0.
\en
If $\G_{1,I}=\emptyset$, then by (\ref{s3.8}) and the condition (\textbf{A2})
we have $\wid{E}\equiv 0$ in a small ball $B(x_0;\delta)\subset\Om_1$.
By \cite[Theorem 6]{C96} we have $\wid{E}\in (H^1(\Om_1))^3$. Thus, by the unique continuation
principle (see \cite[Theorem 2.3]{Ok02}) we have $\wid{E}\equiv0$ in $\Om_1$.
This, together with the transmission condition (\ref{s3.3}) and Holmgren's uniqueness theorem,
implies that $\wid{E}\equiv 0$ in $\Om_h\ba\ov{\Om}_1$.
If $\G_{1,I}\neq\emptyset$, then (\ref{s3.8}) and the boundary condition (\ref{s3.4})
yield that $\nu\times\wid{E}|_{\G_{1,I}}=0$ and $\nu\times\curl\wid{E}|_{\G_{1,I}}=0$.
By unique continuation principle again we have $\wid{E}\equiv0$ in $\Om_1$.
Again, from the transmission condition (\ref{s3.3}) and Holmgren's uniqueness theorem
it follows that $\wid{E}\equiv 0$ in $\Om_h\ba\ov{\Om}_1$.
The uniqueness of solutions is thus proved for both cases.

Now, arguing similarly as in the proof of Theorem 4.1 in \cite{HYZ} or Theorem 3.1
in \cite{HQZ} (see \cite{HYZ,HQZ} for details) we can prove that
the problem (\ref{s3.7}) has a solution $\wid{E}\in X$ satisfying the estimate
\ben
\|\wid{E}\|_X\le C(||f_2||_{Y(\G_0)'}+||f_4||_{L^2_t(\G_{1,I})}
+||\wid{E}_0||_{H(\curl,\Om_h\ba\ov{\Om}_1)}+||E_1||_X)
\enn
with $C$ depending only on $\Om_h$.
Since $\wid{E}=E-E_0-E_1,$ and by taking the infimum over all
$\wid{E}_0\in H(\curl,\Om_h\ba\ov{\Om}_1)$ such that
$\nu\times\wid{E}_0|_{\G_0}=f_1$ and $\nu\times\wid{E}_0|_{\G_h}=0$
and over all $E_1\in H(\curl,\Om_h)$ such that $\nu\times E_1|_{\G_{1,D}}=f_3,$
$\nu\times E_1|_{\G_h}=0$ and $\nu\times E_1|_{\G_{1,I}}\in L^2_t(\G_{1,I})$
the desired estimate follows (on taking into account the definition of the norm
on $Y(\G_0)$ and $Y(\G_{1,D})$).

\end{proof}

\section{The inverse problems}\label{sec4}
\setcounter{equation}{0}

In this section we consider the inverse problems of determining the doubly periodic grating
profile $f$ with its physical property and the refractive index $q$ from a knowledge of
the incident and scattered fields. To this end, we need the free-space quasi-periodic Green's
function
\ben
G_0(x,y)=\frac{1}{8\pi^2}\sum_{n\in\Z^2}\frac{1}{i\beta_n}\exp({i\al_n\cdot(x-y)+i\beta_n|x_3-y_3|})
\enn
provided $\beta_n\neq 0$ for all $n\in\Z^2$ (see \cite{NS}).
In the neighborhood of $x=y$, $G_{0}$ can be represented in the form
$G_{0}(x,y)=\Phi(x,y)+a(x-y)$, where $\Phi(x,y)=\exp({ik_0|x-y|})/(4\pi|x-y|)$
is the fundamental solution to the three-dimensional Helmholtz equation $(\Delta+k^2_0)u=0$ and
$a(x-y)$ is a $C^{\infty}$ function (see \cite{KA1} for the 2D case).
We now introduce the quasi-periodic Green's tensor $\mat{G}_0\in\C^{3\times3}$
for the time-harmonic Maxwell equations:
\be\label{4.1}
\mat{G}_0(x,y)=G_0(x,y)\mat{I}+\frac{1}{k_0^2}\nabla_x\dive_x(G_0(x,y)\mat{I}),\qquad x\neq y,
\en
where $\mat{I}$ is a $3\times3$ identity matrix.
Consider the following incident dipole source located at $z\in\R^3$ with polarization $p$
($|p|=1$):
\ben
E^i(x)=\mat{G}_0(x,z)p,\qquad x\neq z.
\enn
Clearly, we have
\ben
\curl\curl E^{i}(x)-k_{0}^{2}E^{i}(x)=0,\qquad x\neq z.
\enn

\subsection{Mixed reciprocity relations}\label{sec41}

We establish two mixed reciprocity relations for the doubly periodic structure,
which play a key role in the proofs of uniqueness results for the inverse problems.

\begin{lemma}\label{le41}
Assume that $k^2_0q(x)\equiv k^2_1$ is a constant.
For $m=(m_1,m_2)\in\Z^2$ let $E^i(x;m)=(1/k^2_0)\curl\curl[p\exp({i\al_m\cdot x-i\beta_{m}x_3})]$
and let $E(x;m)$ (which is the sum $E^i(x;m)+E^s(x;m)$ in $\Om_0$) be the solution to the
the scattering problem $(\ref{1.1a})-(\ref{RE})$ with $E^i(x)=E^i(x;m).$
On the other hand, define $\wi{\al}:=-\al$ and for $y_0\in\Om_1$ and $r\in\R^3$
let $E^i(x;y_0)=\mat{\wi{G}}_1(x,y_0)r$ and let $\wi{E}(x;y_0)$ solve the scattering problem:
\be\label{l41.a}
\curl\curl\wi{E}-k_{0}^{2}\wi{E}=0 &&\mbox{in}\;\;\Om_{0},\\ \label{l41.b}
\curl\curl\wi{E}-k_{1}^{2}\wi{E}=0 &&\mbox{in}\;\;\Om_{1}\ba\{y_0\},\\ \label{l41.c}
\nu\times\wi{E}|_{+}=\nu\times\wi{E}|_{-},\quad
\nu\times\curl\wi{E}|_{+}=\la_0\nu\times\curl\wi{E}|_{-} &&\mbox{on}\;\;\G_{0},\\ \label{l41.d}
\nu\times\wi{E}=0\quad\mbox{on}\;\;\G_{1,D},\qquad
\nu\times\curl\wi{E}-i\rho(\nu\times\wi{E})\times\nu=0 &&\mbox{on}\;\;\G_{1,I},\\ \label{l41.e}
\wi{E}(x;y_0)=E^i(x;y_0)+\wi{E}^s(x;y_0) &&\mbox{in}\;\;\Om_1\ba\{y_0\},\\ \label{l41.f}
\wi{E}(x;y_0)=\sum_{n\in\Z^2}\wi{E}_n(y_0)\exp(i\wi{\al}_n\cdot x+\wi{\beta}_nx_3)
       &&\mbox{in}\;\;x_3>g_+,\\ \label{l41.g}
\wi{\al}_n\cdot\wi{E}_n+\wi{\beta}_n\cdot\wi{E}_n^{(3)}=0. &&
\en
Here, $\wi{\al}_{n}$ and $\wi{\bt}_{n}$ are defined by
\ben
\wi{\al}_n=(-\al_1+n_1,-\al_2+n_2,0)\quad\mbox{and}\quad
\wi{\beta}_n=\begin{cases}
 \sqrt{k^2_0-|\wi{\al}_n|^2} &\mbox{for}\;|\wi{\al}_n|^2\leq k_0^2,\\
 i\sqrt{|\wi{\al}_n|^2-k^2_0} &\mbox{for}\;|\wi{\al}_n|^2>k_0^2
\end{cases}
\enn
and $\mat{\wi{G}}_1(x,y_0)$ is defined by $(\ref{4.1})$ with $\al_n$ and
$k^2_0$ replaced by $\wi{\al}_n$ and $k_1^2$, respectively. Then we have
\be\label{l41}
r\cdot E(y_0;m)=\frac{8\pi^{2}i}{\la_0}\wi{\beta}_{-m}\wi{E}_{-m}(y_0)\cdot p.
\en
\end{lemma}

\begin{proof}
Note first that $E(x;m)$ and $\wi{E}(x;y_{0})$ are well-defined by the well-posedness
of the direct scattering problem. Applying Green's theorem in $\Om_1\se B(y_{0},\delta)$
and using the fact that contributions of the vertical line integrals cancel out due to the
periodicity, we have
\be\label{l4.1}
0&=&\int_{\Om_{1}\ba B(y_{0},\delta)}\left\{\curl\curl E(x;m)\cdot [\mat{\wi{G}}_{1}(x,y_{0})r]
              -E(x;m)\cdot \curl\curl [\mat{\wi{G}}_{1}(x,y_{0})r]\right\}dx\no\\
 &=&\left[\int_{\G_{0}}-\int_{\G_{1}}+\int_{\pa B_{\de}(y_{0})}\right]
              \left\{\nu\ti\curl E(x)\cdot [\mat{\wi{G}}_{1}(x,y_{0})r]
               -\nu\ti\curl[\mat{\wi{G}}_{1}(x,y_{0})r]\cdot E(x)\right\}ds\no\\
&:=&I_{1}+I_{2}+I_{3},
\en
where $B(y_{0},\delta)$ is a small ball centered at $y_{0}$ with the radius $\delta$
such that $B(y_{0},\de)\subset\Om_{1}$.

We now analyze the asymptotic behavior of $I_{3}$ as $\de\rightarrow 0$.
From the definition that
$\mat{\wi{G}}_1(x,y_0)=\wi{G}_1(x,y_0)\mat{I}+k_1^{-2}\nabla_x\dive_x(\wi{G}_1(x,y_0)\mat{I})$
it follows that
\be\label{l4.2}
I_3&=&\int_{\pa B(y_0,\de)}\left[\nu\ti\curl E(x;m)\cdot k_1^{-2}\nabla\dive[\wi{G}_1(x,y_0)r]
          -\nu\ti\curl[\wi{G}_{1}(x,y_{0})r]\cdot E(x;m)\right]ds\no\\
    &&+\int_{\pa B(y_0,\de)}\nu\ti\curl E(x;m)\cdot r\wi{G}_1(x,y_0)ds\no\\
    &:=&I_4+I_5.
\en
The regularity of $E(x;m)$ and the singularity of $\wi{G}_{1}(x,y_{0})$ at $x=y_0$ imply
that $I_{5}\rightarrow 0$ as $\de\rightarrow 0$.
On the other hand, by the divergence theorem on $\pa B(y_{0},\de)$ it can be seen that
\ben
I_4&=&\int_{\pa B_\de(y_0)}\left[\nu\ti\curl E(x;m)\cdot\frac{1}{k_1^2}\Grad\dive[\wi{G}_1(x,y_0)r]
          -\nu\ti\curl[\wi{G}_1(x,y_0)r]\cdot E(x;m)\right]ds\\
&=&\int_{\pa B_\de(y_0)}\left[\text{Div}(-\nu\ti\curl E(x;m))\frac{1}{k_1^2}\dive[\wi{G}_1(x,y_0)r]
          -\nu\ti\curl[\wi{G}_1(x,y_0)r]\cdot E(x;m)\right]ds\\
&=&\int_{\pa B_\de(y_0)}\left[(\nu\cdot\curl\curl E(x;m))\frac{1}{k_1^2}\dive[\wi{G}_1(x,y_0)r]
          -\nu\ti\curl[\wi{G}_{1}(x,y_{0})r]\cdot E(x;m)\right]ds\\
&=&\int_{\pa B_\de(y_0)}\left[(\nu\cdot E(x;m))\dive[\wi{G}_1(x,y_0)r]
          +\nu\ti E(x;m)\cdot\curl[\wi{G}_{1}(x,y_{0})r]\right]ds\\
&=&\int_{\pa B_{\de}(y_{0})}\left[(\nu\cdot E(x;m))\na\wi{G}_{1}(x,y_{0})
          -\na\wi{G}_{1}(x,y_{0})\ti(\nu\ti E(x;m))\right]ds\cdot r\\
&\to&-r\cdot E(y_0;m)
\enn
as $\delta\to0.$ This combined with (\ref{l4.1}) and (\ref{l4.2}) implies that
\be\label{l4.3}
&&r\cdot E(y_0;m)\no\\
&&=\left(\int_{\G_{0}}-\int_{\G_{1}}\right)
     \left[\nu\ti\curl E(x)\cdot[\mat{\wi{G}}_{1}(x,y_{0})r]
     -\nu\ti\curl[\mat{\wi{G}}_{1}(x,y_{0})r]\cdot E(x)\right]ds.
\en
Similarly, we have on noting the regularity of $\wi{E}^{s}(x;y_{0})$ that
\be\label{l4.4}
\left(\int_{\G_0}-\int_{\G_1}\right)\left[\nu\ti\curl E(x;m)\cdot\wi{E}^s(x;y_0)
     -\nu\ti\curl\wi{E}^s(x;y_0)\cdot E(x;m)\right]ds=0
\en
Combine (\ref{l4.3}) with (\ref{l4.4}) to conclude that
\ben
&&r\cdot E(y_{0};m)\\
&&=\left(\int_{\G_0}-\int_{\G_1}\right)\left[\nu\ti\curl E(x;m)\cdot\wi{E}(x;y_0)
         -\nu\ti\curl\wi{E}(x;y_0)\cdot E(x;m)\right]ds.
\enn
Making use of the boundary conditions on $\G_{j}\;(j=0,1)$ and Green's theorem in
$\Om_h\ba\ov{\Om}_1$ we obtain that
\ben
&&r\cdot E(y_{0};m)\\
&&=\int_{\G_0}\left[\nu\ti\curl E(x;m)|_{-}\cdot\wi{E}(x;y_0)|_{-}
   -\nu\ti\curl\wi{E}(x;y_0)|_{-}\cdot E(x;m)|_{-}\right]ds\\
&&=\frac{1}{\la_{0}}\int_{\G_{0}}\left[\nu\ti\curl E(x;m)|_{+}\cdot\wi{E}(x;y_{0})|_{+}
  -\nu\ti\curl\wi{E}(x;y_{0})|_{+}\cdot E(x;m)|_{+}\right]ds\\
&&=\frac{1}{\la_{0}}\int_{\G_{h}}\left[\nu\ti\curl E(x;m)\cdot\wi{E}(x;y_{0})
  -\nu\ti\curl\wi{E}(x;y_{0})\cdot E(x;m)\right]ds
\enn
Now, by the Rayleigh expansion radiation condition and the divergence-free
property for $E^{s}(x;m)$ and $\wi{E}(x;y_{0})$ we have,
on noting that $\beta_n(\al)=\wi{\beta}_{-n}(\wi{\al})$, that
\ben
\int_{\G_{h}}\left[\nu\ti\curl E^{s}(x;m)\cdot\wi{E}(x;y_{0})
   -\nu\ti\curl\wi{E}(x;y_{0})\cdot E^s(x;m)\right]ds=0.
\enn
This implies that
\ben\label{l4.5}
r\cdot E(y_{0};m)
=\frac{1}{\la_{0}}\int_{\G_{h}}\left[\nu\ti\curl E^{i}(x;m)\cdot\wi{E}(x;y_{0})
  -\nu\ti\curl\wi{E}(x;y_{0})\cdot E^{i}(x;m)\right]ds.
\enn
Insert $E^{i}(x;m)=k^{-2}_0\curl\curl[p\exp({i\al_{m}\cdot x-i\beta_{m}x_3})]$
and $\ds\wi{E}(x;y_{0})=\sum_{n\in\Z^2}\wi{E}_{n}\exp\{i\wi{\al}_{n}\cdot x+i\wi{\beta}_{n}x_3\}$
into the above equation to get
\ben
&&r\cdot E(y_{0};m)
=\frac{i}{\la_0}\sum_{n\in\Z^2}\left\{[\wi{E}_n(y_0)\times e_3]\times(\al_m;-\beta_m)
          -e_3\times[(\wi{\al}_n;\wi{\beta}_n)\times\wi{E}_n(y_0)]\right\}e^{i(\wi{\beta}_n-\beta_m)h}\\
       &&\cdot p_m\int^{2\pi}_{0}\int^{2\pi}_{0}e^{i(\wi{\al}_{n}+\al_{m})\cdot x}dx_1dx_2,
\enn
where $(\vec{a};b)$ is defined as $(\vec{a};b):=\vec{a}+(0,0,b)$ and
$p_{m}=p-[(\al_m;-\beta_m)\cdot p/k_0^2](\al_m;-\beta_m)$.

Finally, use the fact that $\wi{\al}_{n}+\al_{m}=(n+m,0)$, $\wi{\beta}_{-l}(\wi{\al})=\beta_{-l}(\al)$
for all $l\in\Z^2$ and $\ds\int^{2\pi}_{0}\int^{2\pi}_{0}e^{i(n+m,0)\cdot x}dx_1dx_2=0$
for $n+m\neq (0,0)$ to conclude that
\ben
&&r\cdot E(y_{0};m)\\
&&=\frac{4\pi^{2}i}{\la_0}\left\{[\wi{E}_{-m}(y_0)\times e_3]\times(\al_m;-\beta_m)
  -e_3\times[(\wi{\al}_{-m};\wi{\beta}_{-m})\times\wi{E}_{-m}(y_0)]\right\}\cdot p_m\\
&&=-\frac{4\pi^2i}{\la_0}\left\{(\al_m;-\beta_m)\ti[\wi{E}_{-m}(y_0)\times e_3]
  +e_3\ti[(\wi{\al}_{-m};\wi{\beta}_{-m})\ti\wi{E}_{-m}(y_0)]\right\}\cdot p_m\\
&&=-\frac{4\pi^2i}{\la_0}\left\{-\beta_m\wi{E}_{-m}-[(\al_m;-\beta_m)\cdot\wi{E}_{-m}]e_3
  +\wi{E}^{(3)}_{-m}(\wi{\al}_{-m};\wi{\beta}_{-m})-\wi{\beta}_{-m}\wi{E}_{-m}\right\}\cdot p_m\\
&&=-\frac{4\pi^{2}i}{\la_0}\left\{\wi{E}^{(3)}_{-m}(y_0)(\wi{\al}_{-m};\wi{\beta}_{-m})
  -2\wi{\beta}_{-m}\wi{E}_{-m}(y_0)\right\}\cdot p_m\\
&&=\frac{8\pi^{2}i}{\la_{0}}\wi{\bt}_{-m}\wi{E}_{-m}(y_{0})\cdot p,
\enn
where we have used the fact that
\ben
(\wi{\al}_{-m};\wi{\bt}_{-m})\cdot p_{m}&=&(\al_{m};-\bt_{m})\cdot p_{m}=0,\\
(\al_{m};-\bt_{m})\cdot\wi{E}_{-m}&=&(-\wi{\al}_{-m};-\wi{\bt}_{-m})\cdot\wi{E}_{-m}=0.
\enn
This completes the proof.
\end{proof}

If $y_0\in\Om_0$, define the total field $\wi{E}(x;y_0)=\wi{E}^s(x;y_0)
+\mat{\wi{G}}_0(x,y_0)s$ in $\Om_0$, where $\mat{\wi{G}}_{0}(x,y_{0})$ is
an $\wi{\al}$-quasi-periodic Green tensor defined in (\ref{4.1}) with $\al$
replaced with $\wi{\al}.$
Then arguing similarly as in the proof of Lemma \ref{le41} we can prove the following result.

\begin{lemma}\label{le42}
For $m=(m_1,m_2)\in\Z^2$ let $E^i(x;m)=(1/k^2_0)\curl\curl[p\exp({i\al_m\cdot x-i\beta_{m}x_3})]$
and let $E(x;m)$ (which is the sum $E^i(x;m)+E^s(x;m)$ in $\Om_0$) be the solution to the
the scattering problem $(\ref{1.1a})-(\ref{RE})$ with $E^i(x)=E^i(x;m).$
For $y_0\in\Om_0,$ $\wi{\al}=-\al$ and $r\in\R^3$ let $E^i(x;y_0)=\mat{\wi{G}}_0(x,y_0)r$
and let $\wi{E}(x;y_0)$ (which equals to the sum $E^i(x;y_0)+\wi{E}^s(x;y_0)$ in $\Om_0\ba\{y_0\}$)
satisfy the Maxwell equations $\curl\curl\wi{E}-k_0^2\wi{E}=0$ in $\Om_0\ba\{y_0\}$ and
$\curl\curl\wi{E}-k_0^2q\wi{E}=0$ in $\Om_1$ together with the transmission condition
\ben\label{l42.c}
\nu\times\wi{E}|_{+}=\nu\times\wi{E}|_{-},\quad
\nu\times\curl\wi{E}|_{+}=\la_0\nu\times\curl\wi{E}|_{-}\;\;\mbox{on}\;\;\G_0,
\enn
the boundary condition
\ben\label{l42.d}
\nu\times\wi{E}=0\quad\mbox{on}\;\;\G_{1,D},\qquad
\nu\times\curl\wi{E}-i\rho(\nu\times\wi{E})\times\nu=0\;\;\mbox{on}\;\;\G_{1,I}
\enn
and the Rayleigh expansion radiation condition
\ben\label{l42.e}
\wi{E}^s(x;y_0)=\sum_{n\in\Z^2}\wi{E}_n(y_0)\exp(i\wi{\al}_n\cdot x+\wi{\beta}_nx_3)
       \qquad\mbox{for}\;\;x_3>g_+,
\enn
where
\ben\label{l42.g}
\wi{\al}_n\cdot\wi{E}_n+\wi{\beta}_n\cdot\wi{E}_n^{(3)}=0.
\enn
Then we have
\be\label{l42}
r\cdot E(y_0;m)=8\pi^{2}i\wi{\beta}_{-m}\wi{E}_{-m}(y_0)\cdot p.
\en
\end{lemma}

\subsection{Unique determination of the impenetrable profile $f$}\label{sec42}

We now consider the unique determination of the impenetrable grating profile $f$,
assuming that the interface profile $g$ is known and
$k^2_0q(x)\equiv k^2_1$ is a constant. A key ingredient in our proof is
the mixed reciprocity relation for the doubly periodic structure (see Lemma \ref{le41}).

\begin{theorem}\label{thm4.1}
Assume that $\beta_n\neq 0$ for all $n\in\Z^2$, the interface profile $g$ is known and
$k^2_0q(x)\equiv k^2_1$ is a constant. Let $f_1,\,f_2\in C^2(\R^2)$ be $2\pi$-periodic,
let $\rho_1,\,\rho_2$ be two constants and let $h>\max_{x\in\R^2}\{f_1(x),f_2(x)\}$.
If $\nu\times E_{1,m}^s|_{\G_h}=\nu\times E_{2,m}^s|_{\G_h}$ for all incident waves
$E^i_m(x)=(1/k^2_0)\curl\curl[e_l\exp({i\al_m\cdot x-i\beta_{m}x_3})]$ with
$m\in\Z^2$ and $l=1,2,3$, then
\ben
f_1=f_2,\quad\G_{f_1,D}=\G_{f_2,D},\;\;\G_{f_1,I}=\G_{f_2,I},\;\;\rho_1=\rho_2,
\enn
where $e_l$ is the unit vector in the direction $x_l,$ $l=1,2,3.$
Here, $E_{j,m}=E^i_m+E^s_{j,m}$ in $\Om_0$ and $E_{j,m}$ in $\Om_{1f_j}$ are the
unique quasi-periodic solution of the scattering problem $(\ref{1.1a})-(\ref{RE})$
with $E^i=E^i_m$, $\rho=\rho_j$ and $f=f_j$, where
$\Om_{f_j}=\{x\in\R^3\,|\,f_j(x_1,x_2)<x_3<g(x_1,x_2)\},$ $j=1,2.$
\end{theorem}

\begin{proof}
We assume without loss of generality that $f_1\neq f_2$ and there exists a
$z^*=(z_1^*,z_2^*,z_3^*)\in\G_{f_1}$ with $f_1(z_1^*,z_2^*)>f_2(z_1^*,z_2^*)$,
where $\G_{f_j}=\{x\in\R^3\,|\,x_3=f_j(x_1,x_2)\}.$
We choose $\eps>0$ such that $z_\eps:=z^*+\eps e_3\in\Om_{f_1}\cap\Om_{f_2}$.

Let $\wi{E}_{\eps,j}$ be the unique quasi-periodic solution to the scattered
problem (\ref{l41.a})-(\ref{l41.g}) with $y_0=z_\eps,$ $\rho=\rho_j,$ $f=f_j$.
By Lemma \ref{le41} we have
\be\label{thm4.1a}
r\cdot E_{1,m}(z_\eps)&=&\frac{8\pi^2i}{\la_0}
     \wi{\beta}_{-m}\wi{E}_{1,-m}(z_\eps)\cdot e_l,\\ \label{thm4.1b}
r\cdot E_{2,m}(z_\eps)&=&\frac{8\pi^2i}{\la_0}
     \wi{\beta}_{-m}\wi{E}_{2,-m}(z_\eps)\cdot e_l,
\en
where $\wi{E}_{j,n}(z_\eps)$ are the Rayleigh coefficients for $\wi{E}_{\eps,j}$.

On the other hand, from the Rayleigh expansion radiation condition and the assumption that
$\nu\ti E_{1,m}^s|_{\G_h}=\nu\ti E_{2,m}^s|_{\G_h}$
we conclude by the unique continuation principle that $E_{1,m}^s=E_{2,m}^s$ in $\Om_0$.
This, together with the transmission condition on $\G_0$ and Holmgren's uniqueness theorem,
implies that $E_{1,m}=E_{2,m}$ in $\Om_{f_1}\cap\Om_{f_2}$,
so $E_{1,m}(z_\eps)=E_{2,m}(z_\eps)$.
It then follows from (\ref{thm4.1a}) and (\ref{thm4.1b}) that
\ben
\frac{8\pi^{2}i}{\la_0}\wi{\beta}_{-m}\wi{E}_{1,-m}(z_\eps)\cdot e_l
=\frac{8\pi^{2}i}{\la_0}\wi{\beta}_{-m}\wi{E}_{2,-m}(z_\eps)\cdot e_l
\qquad\mbox{or}\quad \wi{E}_{1,-m}(z_\eps)=\wi{E}_{2,-m}(z_\eps).
\enn
Thus, by the Rayleigh expansion radiation condition we have
$\wi{E}_{\eps,1}(x)=\wi{E}_{\eps,2}(x)$ for $x_3>g_+$.
By the unique continuation principle, the transmission condition on $\G_0$
and Holmgren's uniqueness theorem again we obtain that
\ben
\wi{E}_{\eps,1}(x)=\wi{E}_{\eps,2}(x)\quad\mbox{in}\quad\ov{\Om}_0\quad
\mbox{and}\quad\wi{E}^s_{\eps,1}(x)=\wi{E}^s_{\eps,2}(x)\quad\mbox{in}
\quad\ov{\Om_{f_1}\cap\Om_{f_2}}.
\enn

Without loss of generality we may assume that $z^*$ lies on
the coated part of $\G_{f_1}$. Since $z^{*}$ has a positive distance
from $\G_{f_2}$, then the well-posedness of the direct problem implies
that there exists $C>0$ (independent of $\eps$) such that
\ben
|(\nu\ti\curl\wi{E}^s_{\eps,1}-i\rho_1\nu\ti\wi{E}^s_{\eps,1}\ti\nu)(z^*)|
=|(\nu\ti\curl\wi{E}^s_{\eps,2}-i\rho_1\nu\ti\wi{E}^s_{\eps,2}\ti\nu)(z^*)|\le C.
\enn
However, from the boundary condition on $\G_{f_1}$ it is seen that
\ben
&&|(\nu\ti\curl\wi{E}^s_{\eps,1}-i\rho_1\nu\ti\wi{E}^s_{\eps,1}\ti\nu)(z^*)|\\
&&\qquad=|(\nu\ti\curl[\mat{\wi{G}}_1(\cdot,z_\eps)r]
 -i\rho_1\nu\ti[\mat{\wi{G}}_1(\cdot,z_\eps)r]\ti\nu)(z^*)|\rightarrow +\infty
\enn
as $\eps\rightarrow 0$. This is a contradiction, which implies that $f_1=f_2$,
that is, $\Om_{f_1}=\Om_{f_2}$ and $\G_{f_1}=\G_{f_2}$.
Hence, we have $E_{1,m}=E_{2,m}$ in $\Om_{f_1}$.
We claim that $\G_{f_1,D}\cap\G_{f_2,I}$ must be empty
(so $\G_{f_1,D}=\G_{f_2,D}$ and $\G_{f_1,I}=\G_{f_2,I}$)
since, otherwise, a similar argument as below deduces that the total field $E_{1,m}$
vanishes in $\Om_{f_1}$, which is impossible.

Now let $f=f_1=f_2.$ Then by the boundary condition we deduce that
\ben
i(\rho_1-\rho_2)(\nu\times E_{1,m})\times\nu=0\quad\mbox{on}\quad\G_{1,I}.
\enn
If $\rho_1\not=\rho_2$, then the above equation implies that $\nu\times E_{1,m}=0$
on $\G_{1,I}$, so by the boundary condition again $\nu\times\curl E_{1,m}=0$ on $\G_{1,I}$.
Thus, by Holmgren's uniqueness theorem, $E_{1,m}=0$ in $\Om_1$.
By the transmission condition on $\G_0$ and Holmgren's uniqueness theorem again it follows
that $E_{1,m}=E^i_m+E^s_{1,m}=0$ in $\Om_0$, which is a contradiction.
The proof is thus completed.
\end{proof}

\subsection{Unique determination of the refractive index}\label{sec43}

We now consider the inverse problem of recovering the refractive index $q$.
We only consider the case that $\G_{1,I}=\emptyset$, that is, the grating surface $\G_1$ is
a perfect conductor. However, we expect the result to hold in a more
general case by constructing special solutions of the Maxwell equations.
Throughout this section we assume that the transmission constant $\la_0$ is known
and the shape of the grating surfaces $\G_0$ and $\G_1$ is also known and flat, that is,
for two known constants $b>c,$ $g(x')\equiv b$ and $f(x')\equiv c$ for all $x'\in\R^2.$

We have the following global uniqueness result for the inverse problem.

\begin{theorem}\label{thm4.2}
Assume that $q=q_j$ satisfies the conditions $(\textbf{A1})-(\textbf{A3})$
and that $q_j$ depends only on $x_1$ or $x_2$, $j=1,2.$ Let $h>b.$
If
$$
\nu\times E_{1,m}^s|_{\G_h}=\nu\times E_{2,m}^s|_{\G_h}
$$
for all incident waves
$E^i_m(x)=(1/k^2_0)\curl\curl[e_l\exp({i\al_m\cdot x-i\beta_{m}x_3})]$ with
$m\in\Z^2$ and $l=1,2,3$, then we have $q_1=q_2.$
Here, $E_{j,m}=E^i_m+E^s_{j,m}$ in $\Om_0$ and $E_{j,m}$ in $\Om_1$ are the
unique quasi-periodic solution of the scattering problem $(\ref{1.1a})-(\ref{RE})$
with $E^i=E^i_m$ and $q=q_j$, $j=1,2.$
\end{theorem}

\begin{remark}\label{r2}{\rm
Theorem \ref{thm4.2} improves the result in \cite[Theorem 5.4]{HYZ}, where
only the special case $\la_0=1$ is considered and incident waves of the form
(\ref{s43.1}) below are used for all $r\in L^2_t(\G_h)$.
}
\end{remark}

To prove Theorem \ref{thm4.2} we need the following denseness result
which is related to the incident waves of the form
\be\label{s43.1}
E^{i}(x;r)=\int_{\G_h}\mat{\wi{G}}_0(x,y)r(y)ds(y),\qquad x_3<h,
\en
where $r\in L^2_t(\G_h)$.
This result was proved in \cite[Lemma 5.2]{HYZ} for the case $\la_0=1$, and
the general case can be proved similarly (see the proof of Lemma 5.2 in \cite{HYZ}).

\begin{lemma}\label{le44}
The operator $F$ has a dense range in $H^{-1/2}_t(\dive,\G_0)$. Here,
$F:L^2_t(\G_h)\rightarrow H^{-1/2}_t(\dive,\G_0)$ is defined by
$(Fr)(x)=e_3\times\wi{E}(x;r)|_{-}$ on $\G_0$, where $\wi{E}(x;r)$ is the solution
of the scattering problem $(\ref{1.1a})-\ref{RE})$ with the incident wave $E^i(x)=E^i(x;r)$
given by $(\ref{s43.1}).$
\end{lemma}

\textbf{Proof of Theorem \ref{thm4.2}.}
For any $r\in L^2_t(\G_h)$ and $y\in\G_h$ we have by Lemma \ref{le42} that
\be\label{pt42.1}
r(y)\cdot E^s_j(y;m)=8\pi^{2}i\wi{\beta}_{-m}\wi{E}_{j,-m}(y)\cdot e_l,\qquad j=1,2,\;l=1,2,3,
\en
where $\wi{E}_{j,-m}(y)$ are the Rayleigh coefficients of the scattered field
$\wi{E}^{s}_{j}(\cdot;y)$ corresponding to $q=q_j$ and the incident wave
$E^i(x)=\mat{\wi{G}}_0(x,y)r(y).$ It follows from (\ref{pt42.1}) that
\be\label{pt42.2}
\int_{\G_h}r(y)\cdot E^s_j(y;m)ds(y)
=8\pi^{2}i\wi{\beta}_{-m}\int_{\G_h}\wi{E}_{j,-m}(y)ds(y)\cdot e_l.
\en
Denote by $\wi{E}^s_j(x;r)$ and $\wi{E}_j(x;r)$ the scattered and total electric fields,
respectively, corresponding to $q=q_j$ and the incident wave $E^i(x)=E^i(x;r)$, $j=1,2.$
Then from the definition (\ref{s43.1}) of $E^i(x;r)$ it is seen that
\be\label{pt42.3}
\int_{\G_h}\wi{E}_{j,-m}(y)ds(y)\;\;\text{are the Rayleigh coefficients of}\;\;\wi{E}^s_j(x;r).
\en
On the other hand, from the Rayleigh expansion radiation condition and the assumption that
$\nu\times E_1^s(x;m)=\nu\times E_2^s(x;m)$ on $\G_h$
we conclude on using the unique continuation principle that
$E_1^s(x;m)=E_2^s(x;m)$ in $\ov{\Om}_0.$
This, together with (\ref{pt42.2}) and (\ref{pt42.3}), implies that
\ben
\int_{\G_h}\wi{E}_{1,-m}(y)ds(y)=\int_{\G_h}\wi{E}_{2,-m}(y)ds(y).
\enn
From this, the Rayleigh expansion radiation condition and the unique continuation principle
it follows that
\ben
\wi{E}^s_1(x;r)=\wi{E}^s_2(x;r)\quad\mbox{or}\quad
\wi{E}_1(x;r)=\wi{E}_2(x;r)\quad\mbox{in}\;\;\Om_h\ba\Om_1.
\enn
With the help of the transmission conditions on $\G_0$, we get
\ben
\nu\times\wi{E}_1(x;r)|_{-}&=&\nu\times\wi{E}_2(x;r)|_{-}\qquad\quad\;\;\mbox{on}\;\;\G_0,\\
\nu\times\curl\wi{E}_1(x;r)|_{-}&=&\nu\times\curl\wi{E}_2(x;r)|_{-}\qquad\mbox{on}\;\;\G_0.
\enn
Now define $E(x):=\wi{E}_{1}(x;r)-\wi{E}_{2}(x;r)$ in $\ov{\Om}_1$.
Then $E$ satisfies the equation
\ben
\curl\curl E-k^2_0q_2E=k^2_0(q_1-q_2)\wi{E}_1(x;r)\quad\mbox{in}\;\;\Om_1
\enn
and the boundary conditions
\ben
\nu\times E=0,\;\;\nu\times\curl E&=&0\qquad\mbox{on}\;\;\G_0,\\
\nu\times E&=&0\qquad\mbox{on}\;\;\G_1.
\enn
Thus it follows from Green's vector formula that
\be\no
&&k_0^2\int_{\Om_1}(q_1-q_2)\wi{E}_1(x;r)\cdot\ov{E}_2(x)dx\\ \no
&&\qquad\quad=\int_{\Om_1}(\curl\curl E-k^2_0q_2E)\cdot\ov{E}_2(x)dx\\ \no
&&\qquad\quad=\int_{\Om_1}(\curl\curl\ov{E}_2(x)-k^2_0q_2\ov{E}_2(x))\cdot E(x)dx\\
&&\qquad\quad=0 \label{pt42.4}
\en
for any $r\in L^2_t(\G_h),$ where $E_2\in H(\curl,\Om_1)$ satisfies
the Maxwell equation (\ref{1.1b}) with $q=\ov{q}_2$ and the boundary
condition $\nu\times E_2|_{\G_1}=0.$

Now by Lemma \ref{le44} and (\ref{pt42.4}) we obtain that
\be\label{orth}
\int_{\Om_1}(q_1-q_2)E_1(x)\cdot\ov{E}_2(x)dx=0,
\en
where $E_1$ satisfies of the Maxwell equation (\ref{1.1b}) with $q=q_1$
and the boundary condition $\nu\times E_1|_{\G_1}=0.$

Finally, using the orthogonal relation (\ref{orth}) and arguing in exactly the same
way as in the proof of Theorem 5.4 in \cite{HYZ}, we can easily prove that $q_1=q_2$.
The proof is thus completed. $\Box$

\section*{Acknowledgements}

We thank the referee for the valuable comments on the above paper which
helped improve the paper greatly.
This work was supported by the NNSF of China under grant 11071244.

\end{document}